\documentclass[12pt]{amsproc}
\newtheorem{theorem}{\bf Theorem}[section]
\newtheorem{lemma}[theorem]{\bf Lemma}
\newtheorem{proposition}[theorem]{\bf Proposition}

\newtheorem{remark}[theorem]{\sc Remark}

\begin{document}

\title[The commuting probability]{On the commuting probability for subgroups of a finite group}

\author{Eloisa Detomi}
\address{Dipartimento di Ingegneria dell'Informazione - DEI, Universit\`a 
di Padova, Via G. Gradenigo 6/B, 35121 Padova, Italy} 
\email{eloisa.detomi@unipd.it}

\author{Pavel Shumyatsky}
\address{Department of Mathematics, University of Brasilia\\
Brasilia-DF \\ 70910-900 Brazil}
\email{pavel@unb.br}

\thanks{
The first author is a member of GNSAGA (Indam). 
The second author was supported by FAPDF and CNPq.}

\begin{abstract} Let $K$ be a subgroup of a finite group $G$. The probability that an element of $G$ commutes with an element of $K$ is denoted by 
$Pr(K,G)$. Assume that $Pr(K,G)\geq\epsilon$ for some fixed $\epsilon>0$. 
We show that there is a normal subgroup $T\leq G$ and a subgroup $B\leq K$ such that the indexes $[G:T]$ and $[K:B]$ and the order of the commutator subgroup $[T,B]$ are $\epsilon$-bounded. This extends the well known theorem, due to P. M. Neumann, that covers the case where $K=G$.

\noindent We deduce a number of corollaries of this result. A typical application is that if $K$ is the generalized Fitting subgroup $F^*(G)$ then 
$G$ has a class-2-nilpotent normal subgroup $R$ such that both the index $[G:R]$ and the order of the commutator subgroup $[R,R]$ are $\epsilon$-bounded. In the same spirit we consider the cases where $K$ is a term of the lower central series of $G$, or a Sylow subgroup, etc.
\end{abstract}

\maketitle

\section{Introduction} 
The probability that two randomly chosen elements of a finite group $G$ commute is given by $$Pr(G)=\frac{|\{(x,y)\in G\times G\ :\ xy=yx \}|}{|G|^2}.$$ The above number is called the {\it commuting probability} (or 
the {\it commutativity degree}) of $G$. This is a well studied concept. In the literature one can find publications dealing with problems on the set of possible values of $Pr(G)$ and the influence of $Pr(G)$ over the structure of $G$ (see \cite{eberhard,guraro,gustaf,lescot1,lescot2} and references therein). The reader can consult \cite{mann,shalev} and references therein for related developments concerning probabilistic identities in 
groups.

P. M. Neumann \cite{neumann} proved the following theorem (see also \cite{eberhard}).
\begin{theorem}\label{neumann} Let $\epsilon>0$, and let $G$ be a finite group such that $Pr(G)\geq\epsilon$. Then $G$ has a nilpotent normal subgroup $R$ of nilpotency class at most $2$ such that both the index $[G:R]$ 
and the order of the commutator subgroup $[R,R]$ are $\epsilon$-bounded.
\end{theorem}
Throughout the article we use the expression ``$(a,b,\dots)$-bounded" to mean that a quantity is bounded from above by a number depending only on the parameters $a,b,\dots$.

If $K$ is a subgroup of $G$, write $$Pr(K,G)=\frac{|\{(x,y)\in K\times G\ :\ xy=yx \}|}{|K||G|}.$$ This is the probability that an element of $G$ commutes with an element of $K$ (the relative commutativity degree of 
$K$ in $G$).

This notion has been studied in several recent papers (see in particular \cite{lescot3,nath}). Here we will prove the following proposition.

\begin{proposition}\label{main} Let $\epsilon>0$, and let $G$ be a finite 
group having a subgroup $K$ such that $Pr(K,G)\geq\epsilon$. Then there is a normal subgroup $T\leq G$ and a subgroup $B\leq K$ such that the indexes $[G:T]$ and $[K:B]$, and the order of the commutator subgroup $[T,B]$ 
are $\epsilon$-bounded.
\end{proposition}

Theorem \ref{neumann} can be easily obtained from the above result taking 
$K=G$. 

Proposition \ref{main} has some interesting consequences. In particular, we will establish the following results. 

Recall that the generalized Fitting subgroup $F^*(G)$ of a finite group $G$ is the product of the Fitting subgroup $F(G)$ and all subnormal quasisimple subgroups; here a group is quasisimple if it is perfect and its quotient by the centre is a non-abelian simple group. Throughout, by a class-$c$-nilpotent group we mean a  nilpotent group whose nilpotency class is 
at most $c$.

\begin{theorem}\label{fitt} 
Let $G$ be a finite group such that $Pr(F^*(G),G)\geq\epsilon$. Then $G$ has a class-$2$-nilpotent normal subgroup $R$ such that both the index $[G:R]$ and the order of the commutator subgroup $[R,R]$ are $\epsilon$-bounded. 
\end{theorem}
A somewhat surprising aspect of the above theorem is that information on the commuting probability of a subgroup (in this case $F^*(G)$) enables one to draw a conclusion about $G$ as strong as in P. M. Neumann's theorem. Yet, several other results with the same conclusion will be established in this paper. 

Our next theorem deals with the case where $K$ is a subgroup containing $\gamma_i(G)$ for some $i\geq1$. Here and throughout the paper $\gamma_i(G)$ denotes the $i$th term of the lower central series of $G$.

\begin{theorem}\label{main1} Let $\epsilon>0$, and let $K$ be a subgroup of a finite group $G$ containing $\gamma_i(G)$ for some $i\geq1$. Suppose 
that $Pr(K,G)\geq\epsilon$. Then $G$ has a  nilpotent normal subgroup $R$ 
of nilpotency class at most $i+1$ such that both the index $[G:R]$ and the order of $\gamma_{i+1}(R)$ are $\epsilon$-bounded.
\end{theorem}
P. M. Neumann's theorem is a particular case of the above result (take $i=1$).

In the same spirit, we conclude that $G$ has a nilpotent subgroup of $\epsilon$-bounded index if $K$ is a verbal subgroup corresponding to a word implying virtual nilpotency such that $Pr(K,G)\geq\epsilon$. Given a group-word $w$, we write $w(G)$ for the corresponding verbal subgroup of a group $G$, that is the subgroup generated by the values of $w$ in $G$. 
Recall that a group-word $w$ is said to imply virtual nilpotency if every 
finitely generated metabe\-li\-an group $G$ where $w$ is a law, that is $w(G)=1$, has a nilpotent subgroup of finite index. Such words admit several important characterizations (see \cite{black,bume,groves}). 
 In particular, by a result of Gruenberg~\cite{Gr53}, all Engel words imply virtual nilpotency. Burns and Medvedev proved that for any word $w$ implying virtual nilpotency there exist integers $e$ and $c$ depending only on $w$ such that every finite group $G$, in which $w$ is a law, has a class-$c$-nilpotent normal subgroup $N$ such that $G^e\leq N$ \cite{bume}. Here $G^e$ denotes the subgroup generated by all $e$th powers of elements of $G$. Our next theorem provides a probabilistic variation of this result.

\begin{theorem}\label{virtunil} Let $w$ be a group-word implying virtual nilpotency. Suppose that $K$ is a subgroup of a finite group $G$ such that $w(G)\leq K$ and $Pr(K,G)\geq\epsilon$. There is an $(\epsilon,w)$-bounded integer $e$ and a $w$-bounded integer $c$ such that $G^e$ is nilpotent of class at most $c$.
\end{theorem} 
 
We also consider finite groups with a given value of $Pr(P,G)$, where $P$ 
is a Sylow $p$-subgroup of $G$.
\begin{theorem} \label{sylow}
Let $P$ be a Sylow $p$-subgroup of a finite group $G$ such that $Pr(P,G) \ge \epsilon$. Then $G$ has a class-$2$-nilpotent normal $p$-subgroup $L$ 
such that both the index $[P:L]$ and the order of $[L,L]$ are $\epsilon$-bounded. 
\end{theorem}

Once we have information on the commuting probability of all Sylow subgroups of $G$, the result is as strong as in P. M. Neumann's theorem. 

\begin{theorem} \label{allsylow}
Let $\epsilon >0$, and let $G$ be a finite group such that $Pr(P,G) \ge \epsilon$ whenever $P$ is a Sylow subgroup. Then $G$ has a nilpotent normal subgroup $R$ of nilpotency class at most $2$ such that both the index $[G:R]$ and the order of the commutator subgroup $[R,R]$ are $\epsilon$-bounded. 
\end{theorem}

If $\phi$ is an automorphism of a group $G$, the centralizer $C_G(\phi)$ is the subgroup formed by the elements $x\in G$ such that $x^\phi=x$. In the case where $C_G(\phi)=1$ the automorphism $\phi$ is called fixed-point-free. A famous result of Thompson \cite{thompson} says that a finite group admitting a fixed-point-free automorphism of prime order is nilpotent. Higman proved that for each prime $p$ there exists a number $h=h(p)$ depending only on $p$ such that whenever a nilpotent group $G$ admits 
a fixed-point-free automorphism of order $p$, it follows that $G$ has nilpotency class at most $h$ \cite{higman}. Therefore  a finite group admitting a fixed-point-free automorphism of order $p$ is nilpotent of class at most $h$. Khukhro obtained the following ``almost fixed-point-free" generalization of this fact \cite{khukhro}: if a finite group $G$ admits an automorphism $\phi$ of prime order $p$ such that $C_G(\phi)$ has order $m$, then 
$G$ has a nilpotent subgroup of $p$-bounded nilpotency class and $(m,p)$-bounded index. We will establish a probabilistic variation of the above results. Recall that an automorphism $\phi$ of a finite group $G$ is called coprime if  $(|G|,|\phi|)=1$. 

\begin{theorem}\label{auto} Let $G$ be a finite group admitting a coprime 
 automorphism $\phi$ of prime order $p$ such that $Pr(C_G(\phi),G)\geq\epsilon$. Then $G$ has a nilpotent subgroup of $p$-bounded nilpotency class 
and $(\epsilon,p)$-bounded index.
\end{theorem}

An even stronger conclusion will be derived about groups admitting an elementary abelian group of automorphisms of rank at least 2.

\begin{theorem}\label{auto2} Let $\epsilon>0$, and let $G$ be a finite group admitting an elementary abelian coprime group of automorphisms $A$ of 
order $p^2$ such that $Pr(C_G(\phi),G)\geq\epsilon$ for each nontrivial $\phi\in A$. Then $G$ has a class-$2$-nilpotent normal subgroup $R$ such that both the index $[G:R]$ and the order of $[R,R]$ are $(\epsilon,p)$-bounded.
\end{theorem}

Proposition \ref{main}, which is a key result of this paper, will be proved in the next section. The other results will be established in Sections 3--5.

\section{The key result}

A group is said to be a BFC-group if its conjugacy classes are finite and 
of bounded size. A famous theorem of B. H. Neumann says that in a BFC-group the commutator subgroup $G'$ is finite \cite{bhn}. It follows that if $|x^G|\leq m$ for each $x\in G$, then the order of $G'$ is bounded by a number depending only on $m$. A first explicit bound for the order of $G'$ was found by J. Wiegold \cite{wie}, and the best known was obtained in \cite{gumaroti} (see also \cite{neuvoe} and \cite{sesha}). The main technical tools employed in this paper are provided by the recent results \cite{cri,dms,glasgow,dieshu} strengthening B. H. Neumann's theorem.

A well known lemma due to Baer says that if $A,B$ are normal subgroups of 
a group $G$ such that $[A:C_A(B)]\leq m$ and $[B:C_B(A)]\leq m$ for some integer $m\geq1$, then $[A,B]$ has finite $m$-bounded order (see \cite[14.5.2]{Rob}).

We will require a stronger result. Here and in the rest of the paper, given an element $x\in G$ and a subgroup $H\leq G$, we write $x^H$ for the set of conjugates of $x$ by elements from $H$.

\begin{lemma}\label{normal}
 Let $m\geq1$, and let $G$ be a group containing normal subgroups $A,B$ such that $[A:C_A(y)]\leq m$ and $[B:C_B(x)]\leq m$ for all $x\in A$, $y\in B$. Then $[A,B]$ has finite $m$-bounded order.
\end{lemma}

\begin{proof} 
We first prove that, given $x\in A$ and $y\in B$,  the order of $[x,y]$ is $m$-bounded. Let $H=\langle x,y\rangle$. By assumptions, $[A:C_A(y)]\leq m$ and $[B:C_B(x)]\leq m$. Hence there exists an $m$-bounded number $l$ such that $x^l$ and $y^l$ are contained in $Z(H)$ (for example we can take $l=m!$). Let $D=A\cap B \cap H$ and $N=\langle D,x^l,y^l\rangle$. Then $H/N$ is abelian of order at most $l^2$. Both $x$ and $y$ have centralizers of index at most $m$ in $N$. Moreover  every element of $N$ has centralizer of index at most $m$ in $N$. Indeed $|d^N| \le |d^A|\le m $ for every $d \in D \le A\cap B$.  So, every element of $H$ is a product 
of at most $l^2+1$ elements each of which has centralizer of index at most $m$ in $N$. Therefore each element of $H$ has centralizer of $m$-bounded index in $H$. We conclude that $H$ is a BFC-group in which the sizes of 
conjugacy classes are $m$-bounded. Hence $|H'|$ is $m$-bounded and so the 
order of $[x,y]$ is $m$-bounded, too.
 
Now we claim that for every $x \in A$, the subgroup  $[x,B]$ has finite $m$-bounded order. Indeed, $x$ has at most $m$ conjugates $\{x^{b_1}, \dots , x^{b_m} \}$ in $B$, so $[x,B]$  is generated by as most $m$ elements. 
 Let $C$ be a maximal normal subgroup of $B$ contained in $C_B(x)$. Clearly $C$ has $m$-bounded index in $B$ and centralizes $[x,B]$. Thus, the centre of $[x,B]$ has $m$-bounded index in $[x,B]$. It follows from Schur's 
theorem \cite[10.1.4]{Rob} that the derived subgroup of $[x,B]$ has finite $m$-bounded order.  Since $[x,B]$ is generated by as most $m$ elements of $m$-bounded order, we deduce that the order of $[x,B]$ is finite and $m$-bounded. 

Choose $a\in A$ such that $[B:C_B(a)]=\max_{x \in A} [B:C_B(x)]$ and set $n=[B:C_B(a)]$,  where $n \le m$. Let $b_1,\dots, b_n$ be elements of 
$B$ such that $a^B=\{a^{b_1},\dots, a^{b_n}\}$ is the set of (distinct) 
conjugates of $a$ by elements of $B$. Set $U=C_A(b_1,\dots,b_n)$ and note that $U$ has $m$-bounded index in $A$.  
 Given $u\in U$, the elements  $(ua)^{b_1},\dots, (ua)^{b_n}$ form the conjugacy class  $(ua)^B$ because they are all different and their number is the allowed maximum. So, for an arbitrary element $y\in B$ there exists 
$i$ such that $(ua)^y=(ua)^{b_i}=u a^{b_i}$. It follows that $u^{-1}u^y=a^{b_i}a^{-y}$, hence  
  \[[u,y]=a^{b_i}a^{-y} =[a, b_i^{a^{-1}}][ y^{a^{-1}},a] \in [a, B].\] Therefore $[U,B]\leq[a,B]$. Let $a_1,\dots,a_s$ be coset representatives of $U$ in $A$ and note that $s$ is $m$-bounded. As each $[x,B]$ is normal in $B$ and $[U,B]\leq[a,B]$, we deduce  that $[A,B]=[a,B]\prod[a_i,B]$. So $[A,B]$ is a product of $m$-boundedly many subgroups of $m$-bounded order. These subgroups are normal in $B$ and therefore their product has finite $m$-bounded order. 
\end{proof}

In the next lemma $B$ is not necessarily normal. Instead, we require that $B$ is contained in an abelian normal subgroup. Throughout, $\langle H^G\rangle$ denotes the normal closure of a subgroup $H$ in $G$.

\begin{lemma}\label{lem2} Let $m\geq1$, and let $G$ be a group containing 
a normal subgroup $A$ and a subgroup $B$ such that $[A:C_A(y)]\leq m$ and $[B:C_B(x)]\leq m$ for all $x\in A$, $y\in B$. Assume further that $\langle B^G\rangle$ is abelian. Then $[A, B]$ has finite $m$-bounded order.
\end{lemma}

\begin{proof} Without loss of generality we can 
assume that $G=AB$. Set $L=\langle B^G\rangle=\langle B^A\rangle$. 

Let $x\in A$. There is an $m$-bounded number $l$ such that $x$ centralizes $y^l$ for every $y\in B$. Since $L$ is abelian, $[x,y]^i=[x,y^i]$ for 
each $i$ and therefore the order of $[x,y]$ is at most $l$. Thus $[x,B]$ is an abelian subgroup generated by at most $m$ elements of $m$-bounded order, whence $[x,B]$ has finite $m$-bounded order. 

Now we choose $a\in A$ such that $[B:C_B(a)]$ is as big as possible. Let $b_1,\dots,b_{m}$ be elements of $B$ such that $a^B=\{a^{b_1},\dots,a^{b_{m}}\}$. Set $U=C_A(b_1,\dots,b_m)$ and note that $U$ has $m$-bounded 
index in $A$.  Arguing as in the previous lemma, we see that for arbitrary $u\in U$ and $y\in B$, the conjugate $(ua)^y$ belongs to the set 
$\{(ua)^{b_1},\dots, (ua)^{b_m}\}$. Let  $(ua)^y=(ua)^{b_i}$. Then $u^{-1}u^y=a^{b_i}a^{-y}$ and hence $[u,y]=a^{b_i}a^{-y}\in[a, B]$. Therefore $[U,B]\leq[a,B]$. 

Let $V=\cap_{x\in A}U^x$ be the maximal normal subgroup of $A$ contained in $U$. We know that $[V,B]$ has $m$-bounded order, since $[V,B]\le[a,B]$. Denote the index $[A:V]$ by $s$. Evidently, $s$ is $m$-bounded. Let $a_1,\dots,a_s$ be a transversal of $V$ in $A$. As $[V,B] \le L=\langle B^A\rangle$ is abelian, we have 
$$\langle[V,B]^G\rangle=\langle[V,B]^A\rangle=\prod_{i=1}^s[V,B]^{a_i}.$$ 

Thus $[V,L]=[V, B^A] = \langle[V,B]^A\rangle$ is a product of $m$-boundedly many subgroups of $m$-bounded order, and hence it has $m$-bounded order. Write 
$$L=\langle B^A\rangle \leq \langle B^{V a_i} \mid i=1, \dots s \rangle \leq [V,L]\prod_{i=1}^sB^{a_i}.$$ Thus, it becomes clear that $L$ is a product of $m$-boundedly many conjugates of $B$. Say $L $ is a product of $t$ conjugates of $B$. Then, every $y \in L$ can be written as a product of at most $t$ conjugates of elements of $B$ and consequently  $[A: C_A(y)] \le m^t.$  Moreover, as $A$ is normal in $G$ and $|a^B| \le m$ for every $a\in A$,  the conjugacy class $x^L$ of an element  $x \in A$ has size at 
most $m^t$. Now  Lemma \ref{normal} shows that $[A,B] \le [A,L]$ has finite $m$-bounded order.
\end{proof}

We will now show that if $K$ is a subgroup of a finite group $G$ and $N$ is a normal subgroup of $G$, then $Pr(KN/N,G/N)\geq Pr(K,G)$. More precisely, we will establish the following lemma.

\begin{lemma}\label{quoti} Let $N$ be a normal subgroup of a finite group 
$G$, and let $K\leq G$. Then $Pr(K,G)\leq Pr(KN/N,G/N)Pr(N\cap K,N)$.
\end{lemma}

This is an improvement over \cite[Theorem 3.9]{lescot3} where the result was obtained under the additional hypothesis that $N\leq K$.

\begin{proof} In what follows $\bar{G}=G/N$ and $\bar{K}=KN/N$. Write 
$\bar{K_0}$ for the set of cosets $(N\cap K)h$ with $h\in K$. If $S_0=(N\cap K)h\in\bar{K_0}$, write $S$ for the coset $Nh\in\bar{K}$. Of course, we have a natural one-to-one correspondence between $\bar{K_0}$ and $\bar{K}$.

Write
$$|K||G|Pr(K,G)=\sum_{x\in K}|C_G(x)|=\sum_{S_0\in\bar{K_0}}\sum_{x\in S_0}\frac{|C_G(x)N|}{|N|}|C_N(x)| $$

$$\leq\sum_{S_0\in\bar{K_0}}\sum_{x\in S_0}|C_{\bar{G}}(xN)||C_N(x)|=\sum_{S\in\bar{K}}|C_{\bar{G}}(S)|\sum_{x\in S_0}|C_N(x)|=  $$

$$=\sum_{S\in\bar{K}}|C_{\bar{G}}(S)|\sum_{y\in N}|C_{S_0}(y)|.$$

If $C_{S_0}(y)\neq\emptyset$, then there is $y_0\in C_{S_0}(y)$ and so $S_0=(N\cap K)y_0$. Therefore $$C_{S_0}(y)=(N\cap K)y_0\cap C_G(y)=C_{N\cap K}(y)y_0,\text{ whence }|C_{S_0}(y)|=|C_{N\cap K}(y)|.$$ 
Conclude that $$|K||G|Pr(K,G)\leq \sum_{S\in\bar{K}}|C_{\bar{G}}(S)|\sum_{y\in N}|C_{N\cap K}(y)|.$$ Observe that $$\sum_{S\in\bar{K}}|C_{\bar{G}}(S)|=\frac{|K|}{|N\cap K|}\frac{|G|}{|N|}Pr(\bar{K},\bar{G})$$ and $$\sum_{y\in N}|C_{N\cap K}(y)|=|N\cap K||N|Pr(N\cap K,N).$$
It follows that $Pr(K,G)\leq Pr(\bar{K},\bar{G})Pr(N\cap K,N)$, as required.
\end{proof}

The following theorem is taken from \cite{cri}. It plays a crucial role in the proof of Proposition \ref{main}. 
\begin{theorem}\label{cristi} Let $m$ be a positive integer, $G$ a group having a subgroup $K$
such that $|x^G| \le m$ for each $x\in K$, and let $H=\langle K^G\rangle$. Then the order of the commutator subgroup $[H,H]$ is finite and $m$-bounded.
\end{theorem}

A proof of the next lemma can be found in Eberhard \cite[Lemma 2.1]{eberhard}. 

\begin{lemma}\label{lem} 
Let $G$ be a finite group and $X$ a symmetric subset of $G$ containing the identity. Then $\langle X \rangle= X^{3r}$ provided $(r+1)|X| > |G|$. 

\end{lemma}

 We are now ready to prove Proposition \ref{main} which we restate here for the reader's convenience:
\medskip

\noindent{\it  Let $\epsilon>0$, and let $G$ be a finite group having a subgroup $K$ such that $Pr(K,G)\geq\epsilon$. Then there is a normal subgroup $T\leq G$ and a subgroup $B\leq K$ such that the indexes $[G:T]$ and $[K:B]$ and the order of $[T,B]$ are $\epsilon$-bounded.
}

\begin{proof}[Proof of Proposition \ref{main}] Set
 $$X=\{x\in K \mid  |x^G|\leq 2/\epsilon\} \text{ and } B=\langle X\rangle.$$
Note that $K \setminus X=\{ x\in K \mid  |C_G(x)|\leq (\epsilon/2) |G| \}$, whence 
\begin{eqnarray*} 
\epsilon |K||G|  &\le  & | \{ (x,y) \in K\times G \mid xy=yx \} |=\sum_{x\in K}|C_G(x)|\\
&\le & \sum_{x \in X} |G| + \sum_{x  \in K \setminus X} \frac{\epsilon}{2} |G| \\
&\le & |X| |G| +(|K| - |X|)\frac{\epsilon}{2}|G|.  
\end{eqnarray*}
Therefore 
$ \epsilon |K| \le  |X|+  ({\epsilon}/{2}) (|K| - |X|)$,
whence $({\epsilon}/{2}) |K| <  |X|$. 
Clearly, $|B| \ge |X| > ({\epsilon}/{2}) |K|$ and so the index of 
 $B$ in $K$ is at most $2/\epsilon$. 
 As $X$ is symmetric and $(2/\epsilon ) |X| > |K|$, it follows from Lemma 
\ref{lem} that every element of $B$ is a product of at most $6/\epsilon$ elements of $X$.
Therefore $|b^G| \le (2/\epsilon)^{6/\epsilon}$ for every $b \in B$.

 Let $L=\langle B^G\rangle$. Theorem \ref{cristi} tells us that the commutator subgroup $[L,L]$ has $\epsilon$-bounded order. Let us use the bar 
notation for the images of the subgroups of $G$ in $G/[L,L]$. By Lemma \ref{quoti}, 
 $$ Pr(\bar K, \bar G) \ge Pr(K,G)\geq\epsilon. $$ 
 Moreover, 
 $ [\bar K: \bar B] \le [K:B] \le  {\epsilon}/{2} $ and $
 |\bar b^{ \bar G}| \le |b^G| \le (2/\epsilon)^{6/\epsilon} .$
 Thus we can pass to the quotient over $[L,L]$ and assume that $L$ is abelian.   

Now we set
\[ Y= \{ y \in G \mid  |y^K| \le 2/\epsilon\} =
\{ y \in G \mid |C_K(y)| \ge ({\epsilon}/{2}) |K|\}. \]
Note that 
\begin{eqnarray*} 
\epsilon |K||G| & \le & | \{ (x,y) \in K \times G \mid xy=yx \} |\\
&\le & \sum_{y \in Y} |K| + \sum_{y  \in G \setminus Y} \frac{\epsilon}{2} |K| \\
&\le & |Y| |K| +(|G| - |Y|)\frac{\epsilon}{2}|K| \le  |Y| |K| +\frac{\epsilon}{2}|G| |K|.
\end{eqnarray*}
Therefore  $({\epsilon}/{2}) |G| <  |Y|.$

Set $E= \langle Y \rangle$. Thus  $|E| \ge |Y| > ({\epsilon}/{2}) |G|$, 
and so the index of $E$ in $G$ is  at most $2/\epsilon$. As $Y$ is symmetric and $(2/\epsilon ) |Y| > |G|$, it follows from Lemma \ref{lem} that every element of $E$ is a product of at most $6/\epsilon$ elements of $Y$. 
Since $|y^K|\le 2/\epsilon$ for every $y\in Y$, we conclude that $|e^K|\le (2/\epsilon)^{6/\epsilon}$ for every $e\in E$. Let $T$ be the maximal normal subgroup of $G$ contained in $E$. Clearly, the index $[G:T]$ is $\epsilon$-bounded. 
  
 So, now $|b^G| \le (2/\epsilon)^{6/\epsilon} $  for every $b \in B$ and $|e^B| \le (2/\epsilon)^{6/\epsilon}$ for every $e\in T$. 
 As $L$ is abelian, we can apply Lemma \ref{lem2} to conclude that $[T,B]$ has $\epsilon$-bounded order and the result follows. 
\end{proof}

\begin{remark}\label{remark} Under the hypotheses of Proposition \ref{main} the subgroup $N=\langle[T,B]^G\rangle$ has $\epsilon$-bounded order.
\end{remark}
\begin{proof} Since $[T,B]$ is normal in $T$, it follows that there are only boundedly many conjugates of $[T,B]$ in $G$ and they normalize each other. Since $N$ is the product of those conjugates, $N$ has $\epsilon$-bounded order.
\end{proof}

As usual, $Z_i(G)$ stands for the $i$th term of the upper central series of a group $G$.

\begin{remark}\label{remark2} Assume the hypotheses of Proposition \ref{main}. If $K$ is normal, then the subgroup $T$ can be chosen in such a way 
that $K\cap T\leq Z_3(T)$.
\end{remark}
\begin{proof} According to Remark \ref{remark}, $N=\langle[T,B]^G\rangle$ has $\epsilon$-bounded order. Let $B_0=\langle B^G\rangle$ and note that 
$B_0 \le K$ and $[T,B_0]\leq N$. Since the index $[K:B_0]$ and the order of $N$ are $\epsilon$-bounded, the stabilizer in $T$ of the series $$1\leq N\leq B_0\leq K,$$ that is, the subgroup $$H=\{g\in T\ \vert\ [N,g]=1\ \& \ [K,g]\leq B_0\}$$ has $\epsilon$-bounded index in $G$. Note that $K\cap H\leq Z_3(H)$, whence the result.
\end{proof}

\section{Probabilistic almost nilpotency of finite groups}

Our first goal in this section is to furnish a proof of Theorem \ref{fitt}. We restate it here.\medskip

\noindent{\it Let $G$ be a finite group such that $Pr(F^*(G),G)\geq\epsilon$. Then $G$ has a class-$2$-nilpotent normal subgroup $R$ such that both the index $[G:R]$ and the order of the commutator subgroup $[R,R]$ are $\epsilon$-bounded.} 
\medskip

As mentioned in the introduction, the above result yields a conclusion about $G$ which is as strong as in P. M. Neumann's theorem.

\begin{proof}[Proof of Theorem \ref{fitt}.] 
Set $K=F^*(G)$. In view of Proposition \ref{main}  there is a normal subgroup $T\leq G$ and a subgroup $B\leq K$ such that the indexes $[G:T]$ and $[K:B]$, and the order of the commutator subgroup $[T,B]$ are $\epsilon$-bounded. 
 As $K$ is normal in $G$, 
 according to Remark \ref{remark2} the 
subgroup $T$ can be chosen in such a way that $K\cap T\leq Z_3(T)$. By \cite[Corollary X.13.11(c)]{hb3} we have $K\cap T=F^*(T)$. Therefore $F^*(T)\leq Z_3(T)$ and in view of \cite[Theorem X.13.6]{hb3} we conclude that $T=F^*(T)$ and so $T\leq K$. It follows that the index of $K$ in $G$ is $\epsilon$-bounded. By Remark \ref{remark} the subgroup $N=\langle[T,B]^G\rangle$ has $\epsilon$-bounded order. Conclude that $R=\langle B^G\rangle\cap C_G(N)$  has $\epsilon$-bounded index in $G$. Moreover $R$ is nilpotent of class at most 2 and $[R,R]$ has $\epsilon$-bounded order. This completes the proof.
\end{proof}

Now focus on Theorem \ref{main1}, which deals with the case where $\gamma_i(G)\leq K$. Start with a couple of remarks on the result. Let $G$ and $R$ be as in Theorem \ref{main1}. The fact that both the index  $[G:R]$ and the order of $\gamma_{i+1}(R)$ are $\epsilon$-bounded implies that for any $x_1,\dots,x_i\in R$ the centralizer of the long commutator $[x_1,\dots,x_i]$ has $\epsilon$-bounded index in $G$. Therefore there is an $\epsilon$-bounded number $e$ such that $G^e$ centralizes all commutators $[x_1,\dots,x_i]$ where $x_1,\dots,x_i\in R$. Then $G_0=G^e\cap R$ is a nilpotent normal subgroup of nilpotency class at most $i$ with $G/G_0$ of $\epsilon$-bounded exponent (recall that a positive integer $e$ is the exponent of a finite group $G$ if $e$ is the minimal number for which $G^e=1$).

If $G$ is additionally assumed to be $m$-generator for some $m\geq1$, then $G$ has a nilpotent normal subgroup of nilpotency class at most $i$ and $(\epsilon,m)$-bounded index. Indeed, we know that for any $x_1,\dots,x_i\in R$ the centralizer of the long commutator $[x_1,\dots,x_i]$ has $\epsilon$-bounded index in $G$. An $m$-generator group has only $(j,m)$-boundedly many subgroups of any given index $j$ \cite[Theorem 7.2.9]{mhall}. Therefore $G$ has a subgroup $J$ of $(\epsilon,m)$-bounded index that centralizes all commutators $[x_1,\dots,x_i]$ with $x_1,\dots,x_i\in R$. Then $J\cap R$ is a nilpotent normal subgroup of nilpotency class at most $i$ and $(\epsilon,m)$-bounded index in $G$.

These observations are in parallel with Shalev's results on probabilistically nilpotent groups \cite{shalev}.

Our proof of Theorem \ref{main1} requires the following result from \cite{glasgow}.
\begin{theorem}\label{glas} Let $G$ a group such that $|x^{\gamma_k(G)}|\leq n$ for any $x\in G$. Then $\gamma_{k+1}(G)$ has finite $(k,n)$-bounded order.
\end{theorem}

We can now  prove  Theorem \ref{main1}. 
\begin{proof}[Proof of Theorem \ref{main1}] Recall that $K$ is a subgroup 
of the finite group $G$ such that $\gamma_k(G)\leq K$ and $Pr(K,G)\geq\epsilon$. In view of \cite[Theorem 3.7]{lescot3} observe that $Pr(\gamma_k(G),G)\geq\epsilon$. Therefore without loss of generality we can assume that $K=\gamma_k(G)$. 

Proposition \ref{main} tells us that there is a normal subgroup $T\leq G$ and a subgroup $B\leq K$ such that the indexes $[G:T]$ and $[K:B]$ and the order of $[T,B]$ are $\epsilon$-bounded. In particular, $|x^{B}|$ is $\epsilon$-bounded  for every $x\in T$. Since $B$ has  $\epsilon$-bounded index in $K$, we deduce that $|x^{\gamma_k(G)}|$ is $\epsilon$-bounded for every $x\in T$. Now Theorem \ref{glas} implies that $\gamma_{k+1}(T)$ has $\epsilon$-bounded order. Set $R=C_T(\gamma_{k+1}(T))$. It follows that $R$ is as required.
\end{proof} 

Our next goal is a proof of Theorem \ref{virtunil}. As mentioned in the introduction, a group-word $w$ implies virtual nilpotency if every finitely generated metabe\-li\-an group $G$ where $w$ is a law, that is $w(G)=1$, has a nilpotent subgroup of finite index. A theorem, due to Burns and Medvedev, states that for any word $w$ implying virtual nilpotency there exist integers $e$ and $c$ depending only on $w$ such that every finite group $G$, in which $w$ is a law, has a nilpotent of class at most $c$ normal subgroup $N$ with $G^e\leq N$ \cite{bume}. 

\begin{proof}[Proof of Theorem \ref{virtunil}.] Recall that $w$ is a group-word implying virtual nilpotency while $K$ is a subgroup of a finite group $G$ such that $w(G)\leq K$ and $Pr(K,G)\geq\epsilon$. We need to show that there is an $(\epsilon,w)$-bounded integer $e$ and a $w$-bounded integer $c$ such that $G^e$ is nilpotent of class at most $c$.

As in the proof of Theorem \ref{main1} without loss of generality we can assume that $K=w(G)$. Proposition \ref{main} tells us that there is a normal subgroup $T\leq G$ and a subgroup $B\leq K$ such that the indexes $[G:T]$ and $[K:B]$ and the order of the commutator subgroup $[T,B]$ are $\epsilon$-bounded. According to Remark \ref{remark2} the subgroup $T$ can 
be chosen in such a way that $K\cap T\leq Z_3(T)$. In particular $w(T)\leq Z_3(T)$. Taking into account that the word $w$ implies virtual nilpotency, we deduce from the Burns-Medvedev theorem that there are $w$-bounded numbers $i$ and $c$ such that the subgroup generated by the $i$th powers of elements of $T$ is nilpotent of class at most $c$. Recall that the index of $T$ in $G$ is $\epsilon$-bounded. Hence there is an $\epsilon$-bounded integer $e$ such that every $e$th power in $G$ is an $i$th power of an element of $T$. The result follows.
\end{proof} 

If $[x^i,y_1,\dots,y_j]$ is a law in a finite group $G$, then $\gamma_{j+1}(G)$ has $\{i,j\}$-bounded exponent (the case $j=1$ is a well-known result, due to Mann \cite{M}; see \cite[Lemma 2.2]{CS} for the case $j\geq2$). If the $j$-Engel word $[x,y,\dots,y]$, where $y$ is repeated $j$ times, is a law in a finite group $G$, then $G$ has a normal subgroup $N$ such that the exponent of $N$ is $j$-bounded while $G/N$ is nilpotent with $j$-bounded class \cite{bume2}. Note that both words $[x^i,y_1,\dots,y_j]$ and $[x,y,\dots,y]$ imply virtual nilpotency.

Therefore, in addition to Theorem \ref{virtunil}, we deduce

\begin{theorem}\label{exp} Assume the hypotheses of Theorem \ref{virtunil}.
\begin{enumerate}
\item If $w=[x^n,y_1,\dots,y_k]$, then $G$ has a normal subgroup $T$ such that the index $[G:T]$ is $\epsilon$-bounded and the exponent of $\gamma_{k+4}(T)$ is $w$-bounded. 
\item There are $k$-bounded numbers $e_1$ and $c_1$ with the property that if $w$ is the $k$-Engel word, then $G$ has a normal subgroup $T$ such that the index $[G:T]$ is $\epsilon$-bounded and the exponent of $\gamma_{c_1}(T)$ divides $e_1$.
\end{enumerate}
\end{theorem} 
\begin{proof} By \cite[Theorem 3.7]{lescot3}, without loss of generality we can assume that $K=w(G)$. Proposition \ref{main} tells us that there is a normal subgroup $T\leq G$ and a subgroup $B\leq w(G)$ such that the indexes $[G:T]$ and $[w(G):B]$ and the order of $[T,B]$ are $\epsilon$-bounded. Since $K$ is normal in $G$, according to Remark \ref{remark2} the 
subgroup $T$ can be chosen in such a way that $w(G)\cap T\leq Z_3(T)$. If $w=[x^n,y_1,\dots,y_k]$, then $[x^n,y_1,\dots,y_{k+3}]$ is a law in $T$, whence the exponent of $\gamma_{k+4}(T)$ is $w$-bounded. If $w$ is the $k$-Engel word, then the $(k+3)$-Engel word is a law in $T$ and the theorem follows from the Burns-Medvedev theorem \cite{bume2}.
\end{proof}

\section{Sylow subgroups}

As usual, $O_p(G)$ denotes the maximal normal $p$-subgroup of a finite group $G$. For the reader's convenience we restate Theorem \ref{sylow}:
\medskip

\noindent{\it Let $P$ be a Sylow $p$-subgroup of a finite group $G$ such that $Pr(P,G) \ge \epsilon$. Then $G$ has a class-$2$-nilpotent normal $p$-subgroup $L$ such that both the index $[P:L]$ and the order of the commutator subgroup $[L,L]$ are $\epsilon$-bounded. }

\begin{proof}[Proof of Theorem \ref{sylow}] Proposition \ref{main} tells us that there is a normal subgroup $T\leq G$ and a subgroup $B\leq P$ such that the indexes $[G:T]$ and $[P:B]$ and the order of the commutator subgroup $[T,B]$ are $\epsilon$-bounded. In view of Remark \ref{remark} the 
subgroup $N=\langle[T,B]^G\rangle$ has $\epsilon$-bounded order. Therefore $C=C_T(N)$ has $\epsilon$-bounded index in $G$. Set $B_0=B\cap C$ 
and note that $[C,B_0]\leq Z(C)$. It follows that $B_0\leq Z_2(C)$ and we 
conclude that $B_0\leq O_p(G)$. Let $L=\langle {B_0}^G\rangle$. 
 As $B_0 \le L \le O_p(G)$, it is clear that 
 $L$ is contained in $P$ as a subgroup of $\epsilon$-bounded index. Moreover $[L,L]\leq N$ and so the order of $[L,L]$ is $\epsilon$-bounded. Hence the result.
\end{proof}

We will now prove Theorem \ref{allsylow}.

\begin{proof}[Proof of Theorem \ref{allsylow}] Recall that $G$ is a finite group such that $Pr(P,G) \ge \epsilon$ whenever $P$ is a Sylow subgroup. We wish to show that $G$ has a nilpotent normal subgroup $R$ of nilpotency class at most $2$ such that both the index $[G:R]$ and the order of 
the commutator subgroup $[R,R]$ are $\epsilon$-bounded.

For each prime $p\in\pi(G)$ choose a Sylow $p$-subgroup $S_p$ in $G$. Theorem \ref{sylow} shows that $G$ has a normal $p$-subgroup $L_p$ of class at most $2$ such that both $[S_p:L_p]$ and $|[L_p,L_p]|$ are $\epsilon$-bounded. Since the bounds on $[S_p:L_p]$ and $|[L_p,L_p]|$ do not depend on $p$, it follows that there is an $\epsilon$-bounded constant $C$ such that $S_p=L_p$ and $[L_p,L_p]=1$ whenever $p\geq C$. Set $R=\prod_{p\in\pi(G)}L_p$. 
Then all Sylow subgroups of $G/R$ have $\epsilon$-bounded order and therefore the index of $R$ in $G$ is $\epsilon$-bounded. Moreover, $R$ is of class at most $2$ and $|[R,R]|$ is $\epsilon$-bounded, as required. 
\end{proof}

\section{Coprime automorphisms and their fixed points}

If $A$ is a group of automorphisms of a group $G$, we write $C_G(A)$ for the centralizer of $A$ in $G$. The symbol $A^{\#}$ stands for the set of nontrivial elements of the group $A$. 

The next lemma is well-known (see for example \cite[Theorem 6.2.2 (iv)]{go}). In the sequel we use it without explicit references.

\begin{lemma}\label{cc}
Let  $A$ be a group of automorphisms of a finite group $G$ such that $(|G|,|A|)=1$. Then
$C_{G/N}(A)=NC_G(A)/N$ for any $A$-invariant normal subgroup $N$ of $G$.
\end{lemma}

\begin{proof}[Proof of Theorem \ref{auto}.] Recall that $G$ is a finite group admitting a coprime automorphism $\phi$ of prime order $p$ such that 
$Pr(K,G)\geq\epsilon$, where $K=C_G(\phi)$. We need to show that $G$ has a nilpotent subgroup of $p$-bounded nilpotency class and $(\epsilon,p)$-bounded index.

By Proposition \ref{main} there is a normal subgroup $T\leq G$ and a subgroup $B\leq K$ such that the indexes $[G:T]$ and $[K:B]$ and the order of 
the commutator subgroup $[T,B]$ are $\epsilon$-bounded. Let $T_0$ be the maximal $\phi$-invariant subgroup of $T$. Evidently, $T_0$ is normal and the index $[G:T_0]$ is $(\epsilon,p)$-bounded. Since $\langle[T_0,B]^G\rangle\leq\langle[T,B]^G\rangle$, Remark \ref{remark} implies that $M=\langle[T_0,B]^G\rangle$ has $\epsilon$-bounded order.  Moreover, $M$ is $\phi$-invariant. Set $D=C_G(M)\cap T_0$ and $\bar{D}=D/Z_2(D)$, and note that $D$ is $\phi$-invariant. 

In a natural way $\phi$ induces an automorphism of $\bar{D}$ which we will denote by the same symbol $\phi$. 
We note that  $C_{\bar{D}}(\phi)=C_D(\phi) Z_2(D)/Z_2(D) $, so its order is $\epsilon$-bounded because $B\cap D\leq Z_2(D)$. The Khukhro theorem 
\cite{khukhro} now implies that $\bar{D}$ has a nilpotent subgroup of $p$-bounded class and $(\epsilon,p)$-bounded index. Since $\bar{D}=D/Z_2(D)$ and since the index of $D$ in $G$ is $(\epsilon,p)$-bounded, we deduce 
that $G$ has a nilpotent subgroup of $p$-bounded class and $(\epsilon,p)$-bounded index. The proof is complete.
\end{proof}

A proof of the next lemma can be found in \cite{gushu}.

\begin{lemma}\label{eee} If $A$ is a noncyclic elementary abelian $p$-group acting on a finite $p'$-group $G$ in such a way that $|C_G(a)|\leq m$ for each $a\in A^{\#}$, then the order of $G$ is at most $m^{p+1}$.
\end{lemma} 

We will now prove Theorem \ref{auto2}.

\begin{proof}[Proof of Theorem \ref{auto2}] By hypotheses, $G$ is a finite group admitting an elementary abelian coprime group of automorphisms $A$ of order $p^2$ such that $Pr(C_G(\phi),G)\geq\epsilon$ for each $\phi\in A^{\#}$. We need to show that $G$ has a nilpotent normal subgroup $R$ of nilpotency class at most $2$ such that both the index $[G:R]$ and the order of the commutator subgroup $[R,R]$ are $(\epsilon,p)$-bounded. 

Let $A_1,\dots,A_{p+1}$ be the subgroups of order $p$ of $A$ and set $G_i=C_G(A_i)$ for $i=1,\dots,p+1$. According to Proposition \ref{main} for each $i=1,\dots,p+1$ there is a normal subgroup $T_i\leq G$ and a subgroup $B_i\leq G_i$ such that the indexes $[G:T_i]$ and $[G_i:B_i]$ and the order of the commutator subgroup $[T_i,B_i]$ are $\epsilon$-bounded. We let $U_i$ denote the maximal $A$-invariant subgroup of $T_i$ so that each $U_i$ is a normal subgroup of $(\epsilon,p)$-bounded index. The intersection of all $U_i$ will be denoted by $U$. Further, we let $D_i$ denote the maximal $A$-invariant subgroup of $B_i$ so that each $D_i$ has $(\epsilon,p)$-bounded index in $G_i$. Note that a modification of Remark \ref{remark} implies that $N_i=\langle[U_i,D_i]^G\rangle$ is $A$-invariant and has $\epsilon$-bounded order. It follows that the order of $N=\prod_iN_i$ is $(\epsilon,p)$-bounded. Let $V$ denote the minimal ($A$-invariant) normal subgroup of $G$ containing all $D_i$ for $i=1,\dots,p+1$. It is easy to see that $[U,V]\leq N$.

Obviously, $U$ has $(\epsilon,p)$-bounded index in $G$. Let us check that 
this also holds with respect to $V$. Let $\bar{G}=G/V$. Since $V$ contains $D_i$ for each $i=1,\dots,p+1$ and since $D_i$ has $(\epsilon,p)$-bounded index in $G_i$, we conclude that the image of $G_i$ in $\bar{G}$ has $(\epsilon,p)$-bounded order. Now Lemma \ref{eee} tells us that the order of $\bar{G}$ is $(\epsilon,p)$-bounded and we conclude that indeed $V$ has $(\epsilon,p)$-bounded index in $G$. Also note that since $N$ has $(\epsilon,p)$-bounded order, $C_G(N)$ has $(\epsilon,p)$-bounded index in 
$G$. Let $$R=U\cap V\cap C_G(N).$$ Then $R$ is as required since the subgroups $U,V,C_G(N)$ have $(\epsilon,p)$-bounded index in $G$ while $[R,R]\leq N\leq C_G(R)$. The proof is complete.
 \end{proof}

\end{document}